\documentclass[11pt]{amsart}
\usepackage{amsmath,amsthm,amsfonts,amssymb,amscd}
\usepackage[frenchb]{babel}
\usepackage[T1]{fontenc}

\newcommand{\al}{\alpha}
\newcommand{\bt}{\beta}
\newcommand{\Dt}{\Delta}
\newcommand{\io}{\iota}
\newcommand{\la}{\lambda}
\newcommand{\La}{\Lambda}
\newcommand{\ph}{\varphi}
\newcommand{\si}{\sigma}
\newcommand{\om}{\omega}
\newcommand{\su}{\subset}
\newcommand{\sm}{\setminus}
\newcommand{\ti}{\times}
\newcommand{\ag}{\langle}
\newcommand{\ad}{\rangle}
\newcommand{\vi}{\emptyset}
\newcommand{\pa}{\partial}
\newcommand{\ov}{\overline}
\newcommand{\wh}{\widehat}
\newcommand{\ii}{\infty}
\newcommand{\C}{\mathcal C}
\newcommand{\U}{\mathcal U}
\newcommand{\V}{\mathcal V}
\newcommand{\W}{\mathcal W}
\newcommand{\CX}{\mathcal X}
\newcommand{\cH}{\check{H}^*}
\newcommand{\fS}{\mathfrak S}
\newcommand{\ft}{\mathfrak t}
\newcommand{\fT}{\mathfrak T}
\newcommand{\bS}{\mathbf S}

\DeclareMathOperator{\St}{St}
\DeclareMathOperator{\Tr}{Tr}
\DeclareMathOperator{\tr}{tr}
\DeclareMathOperator{\Hom}{Hom}

\newtheorem{lem}{Lemme}
\newtheorem{defi}{D\'efinition}

\begin{document}

\title[Conjecture de Schauder]{Une g\'en\'eralisation de la conjecture de point fixe de Schauder}

\author{Robert Cauty}

\address{Universit\'e Paris 6, Institut de math\'ematiques, case 247,4 place Jussieu, 75252 Paris cedex 05}

\email{cauty@math.jussieu.fr}



\begin{abstract} We prove the following generalisation of Schauder's fixed point conjecture: Let $C_1,\dots,C_n$ be convex subsets of a Hausdorff topological vector space. Suppose that the $C_i$ are closed in $C=C_1\cup\dots\cup C_n$. If $f:C\to C$ is a continuous function whose image is contained in a compact subset of $C$, then its Lefschetz number $\La(f)$ is defined, and if $\La(f)\ne0$, then $f$ has a fixed point. \end{abstract}

\maketitle

\section{Introduction}
Soit $X$ un espace topologique s\'epar\'e. Une fonction continue $f:X\to X$ est dite compacte si $f(X)$ est contenu dans un sous-ensemble compact de $X$. Pour d\'efinir, lorsque cela est possible, le nombre de Lefschetz $\La(f)$ d'une telle fonction, nous utilisons la cohomologie d'Alexander-Spanier \`a coefficients rationnels.

Nous d\'emontrerons la g\'en\'eralisation suivante de la conjecture de point fixe de Schauder.

\medskip \noindent
{\bf Th\'eor\`eme.} {\it Soit $C=C_1\cup\dots\cup C_n$, o\`u $C_1,\dots,C_n$ sont des sous-ensembles convexes d'un e.v.t. s\'epar\'e $E$ qui sont ferm\'es dans $C$. Si $f:C\to C$ est une fonction continue compacte, alors $\La(f)$ est d\'efini. Si $\La(f)\ne0$, $f$ a un point fixe.}

\medskip
La d\'emonstration suit le sch\'ema que nous avons expliqu\'e dans \cite{Ca3} pour les espaces ULC. Malheureusement, les r\'esultats expos\'es dans \cite{Ca3} ne sont pas directement applicables ici, car nous ne savons pas si $C$ est ULC, m\^eme dans le cas o\`u il est m\'etrisable et $C=C_1\cup C_2$. Le probl\`eme vient de l'absence d'un th\'eor\`eme d'addition pour les espaces ULC. Soient $X_1$ et $X_2$ des ferm\'es d'un espace m\'etrisable $X$ tels que $X=X_1\cup X_2$. Si $X_1$, $X_2$ et $X_0=X_1\cap X_2$ sont ULC, $X$ est-il aussi ULC ? La r\'eponse affirmative n'est connue que dans le cas o\`u $X_0$ est un r\'etracte de voisinage de $X$ (voir \cite{Hi}), mais cette condition n'est pas toujours v\'erifi\'ee. Il existe des ensembles convexes m\'etrisables $C_0\su K$ tels que $C_0$ soit ferm\'e dans $K$, que $K$ soit un r\'etracte absolu, mais que $C_0$ n'en soit pas un (voir \cite{Ca1}); alors $C_0$ n'est pas un r\'etracte de voisinage de $K$. Si $C=C_1\cup C_2$ o\`u $(C_1,C_0)$ et $(C_2,C_0)$ sont des copies de $(K,C_0)$, $C$ est-il ULC ?

Comme dans tous les probl\`emes de point fixe de ce type, la partie la plus d\'elicate de la d\'emonstration est le cas o\`u $C$ est m\'etrisable. Dans ce cas, $C$ a le type d'homotopie du nerf du syst\`eme $\{C_1,\dots,C_n\}$, ce qui nous permet d'utiliser l'homologie singuli\`ere pour d\'efinir $\La(f)$, et une variante de la technique, introduite implicitement dans \cite{Ca2}, d'\og approximation\fg des fonctions continues par des morphismes de cha\^\i nes d\'efinis sur le complexe $S(C,\Bbb Q)$ des cha\^\i nes singuli\`eres de $C$ \`a coefficients rationnels. Toutefois, les arguments de \cite{Ca2} ne peuvent \^etre appliqu\'es directement ici car, bien que les $C_i$ soient des r\'etractes absolus de voisinage alg\'ebriques\footnote{La d\'efinition de ces espaces sera rappel\'ee \`a la section 3}, nous ne savons pas s'il en est de m\^eme de $C$. Cela vient du fait qu'on ne sait pas si le th\'eor\`eme d'addition est vrai pour les r\'etractes absolus de voisinage alg\'ebriques: si $X_1$ et $X_2$ sont des ferm\'es de l'espace m\'etrisable $X$ tels que $X=X_1\cup X_2$ et que $X_1$, $X_2$ et $X_0=X_1\cap X_2$ sont des r\'etractes absolus de voisinage alg\'ebriques, $X$ est-il aussi un r\'etracte absolu de voisinage alg\'ebrique ? Pour contourner cette difficult\'e, nous d\'emontrerons \`a la section 3 un r\'esultat d'invariance de l'homologie par \og petites d\'eformations\fg\  du complexe de cha\^\i nes $S(C,\Bbb Q)$.

Le passage du cas m\'etrisable au cas g\'en\'eral suit le sch\'ema, expos\'e dans \cite{Ca3}, utilisant l'espace convexe libre associ\'e \`a un compact $X$, et la repr\'esentation de ce compact comme limite d'un syst\`eme projectif particulier de compacts m\'etrisables.

Bien que l'utilisation des coefficients rationnels suffise pour la d\'emonstra\-tion du th\'eor\`eme, certains r\'esultats auxiliares sont vrais -avec les m\^emes d\'emon\-strations- pour tout anneau de coefficients, et nous les formulerons et d\'emontrerons dans cette g\'en\'eralit\'e.

\section{Pr\'eliminaires}
Dans tout cet article, $R$ d\'esigne un anneau unitaire. Tous les complexes de cha\^\i nes $M=\{M_p\}$ de $R$-modules que nous utiliserons sont positifs ($M_p=0$ pour $p<0$) et augment\'es sur $R$. {\it Tous les morphismes de cha\^\i nes entre deux tels complexes seront suppos\'es pr\'eserver l'augmentation}. Un complexe de cha\^\i nes est acyclique si son homologie r\'eduite est triviale. Un module gradu\'e est de type fini s'il est engendr\'e par un nombre fini d'\'el\'ements.

Pour $p\ge0$, soit $\Dt_p$ l'ensemble des points $(t_0,\dots,t_p)$ de $\Bbb R^{p+1}$ tels que $t_i\ge0$ pour $0\le i\le p$ et $\sum_{i=0}^pt_i=1$. Si $X$ est un espace topologique, un $p$-simplexe singulier de $X$ est une fonction continue $\si:\Dt_p\to X$. Les $0$-simplexes singuliers s'identifient naturellement aux points de $X$. Nous notons $S(X,R)$ le complexe des cha\^\i nes singuli\`eres de $X$ \`a coefficients $R$; le module des $p$-cha\^\i nes $S_p(X,R)$ de ce complexe est le module libre engendr\'e par les $p$-simplexes singuliers de $X$. Nous notons $H(X,R)$ le module gradu\'e d'homologie de ce complexe. Nous identifions chaque simplexe singulier $\si$ de $X$ \`a l'\'el\'ement $1\cdot\si$ de $S(X,R)$. Si $A$ est un sous-ensemble de $X$, nous identifions naturellement $S(A,R)$ \`a un sous-complexe de $S(X,R)$. Le support d'une cha\^\i ne $c\in S(X,R)$ est not\'e $||c||$; c'est le plus petit sous-ensemble $C$ de $X$ tel que $c\in S(C,R)$. En particulier, l'image du simplexe singulier $\si$ est not\'ee $||\si||$. Si $\si$ et $\tau$ sont deux simplexes singuliers de $X$, la notation $\si<\tau$ signifie que $\si$ est une face propre de $\tau$. Si $f:X\to Y$ est une fonction continue, nous notons $f_\#:S(X,R)\to S(Y,R)$ et $f_*:H(X,R)\to H(Y,R)$ les morphismes induits par $f$. Si $\ph$ est un morphisme de cha\^\i nes, nous notons $\ph_*$ le morphisme qu'il induit sur l'homologie.

Si $\C$ est une famille de sous-ensembles d'un espace topologique $X$, nous notons $S(X,\C,R)$ le sous-complexe de $S(X,R)$ engendr\'e par les simplexes singuliers dont l'image est contenue dans un \'el\'ement de $\C$. Le module gradu\'e d'homologie de $S(X,\C,R)$ est not\'e $H(X,\C,R)$. Il est connu (voir \cite{Sp}, \S4.4) que si $\U$ est un recouvrement ouvert de $X$, il existe un morphisme de cha\^\i nes $Sd_\U:S(X,R)\to S(X,\U,R)$ tel que $Sd_\U(c)=c$ pour $c\in S(X,\U,R)$ et que $||Sd(c)||\su ||c||$ pour toute cha\^\i ne $c$. En outre, notant $\io_\U$ l'inclusion de $S(X,\U,R)$ dans $S(X,R)$, il existe une homotopie $h:S(X,R)\to S(X,R)$ entre l'identit\'e et $\io_\U\circ Sd_\U$ telle que $||h(c)||\su||c||$ pour toute cha\^\i ne $c$. Un tel morphisme $Sd_\U$ est appel\'e un op\'erateur de subdivision. Les homomorphismes $\io_{\U*}:H(X,\U,R)\to H(X,R)$ et $Sd_{\U*}:H(X,R)\to H(X,\U,R)$ induits par $\io_\U$ et $Sd_\U$ respectivement sont des isomorphismes inverses l'un de l'autre.

Pour tout espace topologique $X$, nous notons $\cH(X,R)$ le module de cohomologie d'Alexander-Spanier de $X$ \`a coefficients $R$. Si $f:X\to Y$ est une fonction continue, nous notons $f^*:\cH(Y,R)\to \cH(X,R)$ l'homomorphisme qu'elle induit.

Si $K$ est un complexe simplicial, nous notons $|K|$ sa r\'ealisation g\'eom\'etrique. Si $s$ et $s'$ sont des simplexes de $K$, la notation $s\le s'$ (resp. $s<s'$) signifie que $s$ est une face de $s'$ (resp. une face propre de $s'$). Pour tout sommet $v$ de $K$, nous notons $\St(v,K)$ l'\'etoile ouverte de $v$ dans $|K|$. Si les sommets $v_0,\dots,v_p$ de $K$ d\'eterminent un simplexe, nous notons $\ag v_0,\dots,v_p\ad$ ce simplexe. Le barycentre d'un simplexe $s$ est not\'e $[s]$. Nous notons $K'$ la subdivision barycentrique de $K$; ses simplexes sont de la forme $\ag[s_0],\dots,[s_q]\ad$, o\`u $s_0,\dots,s_q$ sont des simplexes de $K$ tels que $s_0\le\dots\le s_q$. Nous identifions naturellement $|K|$ et $|K'|$. Pour tout simplexe $s$ de $K$, nous notons $\Tr s$ le sous-complexe de $K'$ form\'e des simplexes $\ag [s_0],\dots,[s_q]\ad$ tels que $s\le s_0$. $\Tr s$ est un c\^one de sommet $[s]$ ou est r\'eduit \`a ce point, donc $|\Tr s|$ est contractile.

Si $\C=\{C_1,\dots,C_p\}$ est une famille de sous-ensembles d'un espace $X$, le sommet du nerf de $\C$ correspondant \`a l'\'el\'ement $C_i\ne\vi$ sera identifi\'e \`a $i$. Si $s$ est un sous-ensemble non vide de $\{1,\dots,n\}$, nous posons $C_s=\bigcap_{i\in s}C_i$. Les simplexes du nerf de $\C$ sont les sous-ensembles non vides $s$ de $\{1,\dots,n\}$ tels que $C_s\ne\vi$.

\begin{lem} Si l'espace $C$ du th\'eor\`eme est normal, alors il a le type d'homotopie du nerf de $\C$. \end{lem}

\begin{proof} Soit $N$ le nerf de $\C$, et soit $m$ sa dimension. Si $s$ est un simplexe de $N$, nous notons $d(s)$ sa dimension. 

Construisons une fonction continue $f:C\to|N|$ telle que $f(C_s)\su|\Tr s|$ pour tout simplexe $s$ de $N$. Pour $0\le p\le m$, soit $C^{(p)}=\bigcup_{d(s)\ge p}C_s$. Alors $C^{(0)}=C$ et $C^{(m)}$ est la r\'eunion disjointe des $C_s$ tels que $d(s)=m$. Si $d(s)=m$, alors $\Tr s$ est le $0$-simplexe $\ag[s]\ad$, et nous d\'efinissons la restriction $f_m$ de $f$ \`a $C^{(m)}$ en posant $f_m(C_s)=[s]$ pour tout $m$-simplexe $s$. Soit $0\le p<m$, et supposons construite la restriction $f_{p+1}$ de $f$ \`a $C^{(p+1)}$. Si $d(s)=p$, alors $C_s\cap C^{(p+1)}=\bigcup_{s<s'}C_{s'}$. Si $s<s'$, alors $\Tr s'\su\Tr s$, donc $f_{p+1}(C_s\cap C^{(p+1)})\su|\Tr s|$. $|\Tr s|$, \'etant un poly\`edre fini contractile, est hom\'eomorphe \`a un r\'etracte d'un cube, donc est un extenseur absolu pour la classe des espaces normaux d'apr\`es le th\'eor\`eme de Tietze-Urysohn. Il existe donc une fonction continue $f_s:C_s\to|\Tr s|$ telle que $f_s|C_s\cap C^{(p+1)}=f_{p+1}|C_s\cap C^{(p+1}$. Si $s$ et $s'$ sont deux $p$-simplexes distincts, alors $C_s\cap C_{s'}\su C^{(p+1)}$, ce qui nous permet de d\'efinir la restriction $f_p$ de $f$ \`a $C^{(p)}$ par
$$
f_p(x)=\begin{cases}f_{p+1}(x)&\text{si $x\in C^{(p+1)}$}\\
f_s(x)&\text{si $x\in C_s$ avec $d(s)=p$}. \end{cases}
$$

Construisons aussi une fonction continue $g:|N|\to C$ telle que, pour tout simplexe $s'=\ag[s_0],\dots,[s_q]\ad$ de $N'$, $g(|s'|)\su C_{s_0}$. Si $s'=\ag[s_0]\ad$ est un $0$-simplexe de $N'$, prenons un point $x_{s_0}\in C_{s_0}$ et posons $g([s_0])=x_{s_0}$. Soit $0<q\le k$, et supposons d\'efinie la restriction de $g$ au $(q-1)$-squelette de $|N'|$. Soit $s'=\ag[s_0],\dots,[s_q]\ad$ un $q$-simplexe de $N'$. Pour toute face $s''$ de $s'$, nous avons $g(|s''|)\su C_{s_0}$, donc $g$ envoie le bord de $|s'|$ dans $C_{s_0}$. Comme $C_{s_0}$ est convexe, donc contractile, nous pouvons prolonger $g$ \`a $|s'|$ de fa\c con que $g(|s'|)\su C_{s_0}$. Pour tout simplexe $s$ de $N$, nous avons $ g(|\Tr s|)\su C_s$ car, si $s'=\ag[s_0],\dots,[s_q]\ad$ est un simplexe de $\Tr s$, alors $s\le s_0$, donc $g(|s'|)\su C_{s_0}\su C_s$.

Pour tout simplexe $s$ de $N$, nous avons $g\circ f(C_s)\su g(|\Tr s|)\su C_s$. Comme les $C_s$ sont convexes, nous pouvons d\'efinir une homotopie $h$ entre l'identit\'e de $C$ et $g\circ f$ par $h(x,t)=(1-t)x+tgf(x)$.

Pour tout simplexe $s$ de $N$, nous avons $f\circ g(|\Tr s|)\su f(C_s)\su|\Tr s|$. Nous ach\`everons de montrer que $f$ et $g$ sont des \'equivalences homotopiques en construisant une homotopie $k:|N|\ti I\to|N|$ entre l'identit\'e et $f\circ g$ telle que $k(|\Tr s|\ti I)\su|\Tr s|$ pour tout simplexe $s$ de $N$. Pour $0\le q\le m$, soit $T_q=\bigcup_{q\le d(s)}|\Tr s|$; alors $T_0=|N'|$. Si $d(s)=m$, alors $\Tr s=[s]$, et la restriction de $k$ \`a $|\Tr s|\ti I$ est d\'efinie par $k([s],t)=[s]$ pour tout $t$. Soit $0\le q<m$, et supposons d\'efinie la restriction $k_{q+1}$ de $k$ \`a $T_{q+1}\ti I$. Si $s$ est un $q$-simplexe de $N$, alors $|\Tr s|\cap T_{q+1}=\bigcup_{s<s'}|\Tr s'|$, donc $k_{q+1}((|\Tr s|\cap T_{q+1})\ti I)\su|\Tr s|$, et la contractilit\'e du poly\`edre $|\Tr s|$ nous permet de trouver une homotopie $k_s:|\Tr s|\ti I\to|\Tr s|$ entre l'identit\'e et la restriction de $f\circ g$ telle que $k_s|(|\Tr s|\cap T_{q+1})\ti I=k_{q+1}|(|\Tr s|\cap T_{q+1}\ti I$. Nous pouvons d\'efinir la restriction $k_q$ de $k$ \`a $T_q\ti I$ par
$$
k_q(x,t)=\begin{cases}k_{q+1}(x,t)&\text{si $x\in T_{q+1}$}\\
k_s(x,t)&\text{si $x\in |\Tr s|$ avec $d(s)=q$}. \end{cases}
$$
\end{proof}

Il r\'esulte du lemme 1 que, si $C$ est normal, alors $\cH(C,\Bbb Q)$ est de type fini, donc le nombre de Lefschetz $\La(f)$ est d\'efini pour toute fonction continue $f:C\to C$. En outre, $\cH(C,\Bbb Q)$ est alors naturellement isomorphe \`a $\Hom(H(C,\Bbb Q),\Bbb Q)$, ce qui permet d'utiliser l'homologie singuli\`ere pour calculer $\La(f)$.

Si $\C=\{C_1,\dots,C_n\}$ et $\C'=\{C'_1,\dots,C'_n\}$ sont deux familles de sous-ensembles de $X$ et $X'$ respectivement, nous dirons que les nerfs de $\C$ et $\C'$ sont naturellement isomorphes si, pour tout sous-ensemble non vide $s$ de $\{1,\dots,n\}$, les conditions $C_s\ne\vi$ et $C'_s\ne\vi$ sont \'equivalentes.

\begin{lem} Soient $\C=\{C_1,\dots,C_n\}$ et $\C'=\{C'_1,\dots,C'_n\}$ deux familles de sous-ensembles convexes des e.v.t. $E$ et $E'$ respectivement, et soient $C=C_1\cup\dots\cup C_n$ et $C'=C'_1\cup\dots\cup C'_n$. Supposons les $C_i$ ferm\'es dans $C$ et les $C'_i$ fem\'es dans $C'$. Soit $q:C\to C'$ une fonction continue telle que $q(C_i)\su C'_i$ pour tout $i$. Si $C$ et $C'$ sont normaux et si les nerfs de $\C$ et $\C'$ sont naturellement isomorphes, alors $q^*:\cH(C',R)\to\cH(C,R)$ est un isomorphisme. \end{lem}

\begin{proof} Soit $N$ le nerf commun \`a $\C$ et $\C'$. Comme dans le lemme 1, construisons une fonction continue $f':C'\to|N'|$ telle que $f'(C'_s)\su|\Tr s|$ pour tout simplexe $s$ de $N$. La fonction $f=f'\circ q:C\to|N|$ v\'erifie alors $f(C_s)\su|\Tr s|$ pour tout simplexe $s$ de $N$. L'argument du lemme 1 montre que $f'$ et $f$ sont des \'equivalences homotopiques, donc ${f'}^*:\cH(|N|,R)\to \cH(C',R)$ et $f^*:\cH(|N|,R)\to\cH(C,R)$ sont des isomorphismes. Puisque $f^*=q^*\circ {f'}^*$, $q^*$ est aussi un isomorphisme. \end{proof}

Nous avons aussi besoin d'une propri\'et\'e d'invariance de l'homologie singuli\`ere par les \og petites d\'eformations\fg.

\begin{lem} Si l'espace topologique $X$ a le type d'homotopie d'un complexe simplicial, alors il existe un recouvrement ouvert $\U$ de $X$ ayant la propri\'et\'e suivante: si $\V$ est un recouvrement ouvert de $X$ et si $\ph:S(X,\V,R)\to S(X,R)$ est un morphisme de cha\^\i nes tel que, pour tout simplexe singulier $\si$ de $S(X,\V,R)$, il existe un \'el\'ement de $\U$ contenant $||\si||$ et $||\ph(\tau)||$ pour toute face $\tau$ de $\si$, alors $\ph_*\circ Sd_{\V*}(a)=a$ pour tout $a\in H(X,R)$. \end{lem}

\begin{proof} Par hypoth\`ese, il existe un complexe simplicial $N$ et une \'equivalence homotopique $q:X\to|N|$. Soit $\U_0$ le recouvrement ouvert de $|N|$ form\'e des ensembles $\St(v,N)$, o\`u $v$ parcourt l'ensemble $N^{(0)}$ des sommets de $N$, et soit $\U=q^{-1}(\U_0)$ le recouvrement ouvert de $X$ form\'e des ensembles $q^{-1}(\St(v,N))$, $v\in N^{(0)}$.

Soient $\V$ un recouvrement ouvert de $X$ et $\ph:S(X,\V,R)\to S(X,R)$ un morphisme de cha\^\i nes tel que, pour tout simplexe singulier $\si$ de $S(X,\V,R)$, il existe un \'el\'ement $q^{-1}(\St(v,N))$ de $\U$ contenant $||\si||$ et $||\ph(\tau)||$ pour toute face $\tau$ de $\si$, et l'ensemble de ces $v\in N^{(0)}$ est un simplexe $s_\si$ de $N$. L'ensemble $\St(s_\si,N)=\bigcap_{v\in s_\si}\St(v,N)$ est contractile. Si $\tau$ est une face de $\si$, alors $s_\si\su s_\tau$, donc $\St(s_\tau,N)\su\St(s_\si,N)$, et $\St(s_\si,N)$ contient $q(||\tau||)\cup||q_\#\circ\ph(\tau)||$ pour toute face $\tau$ de $\si$. L'acyclicit\'e des ensembles $\St(s_\si,N)$ permet alors de construire, par r\'ecurrence sur la dimension de $\si$, un \'el\'ement $h(\si)\in H(\St(s_\si,N),R)$ tel que $\pa h(\si)+h\pa(\si)=q_\#(\si)-q_\#\circ\ph(\si)$. Les morphismes de cha\^\i nes $q_\#\circ\io_\V$ et $q_\#\circ\ph$ sont donc homotopes, et nous avons $q_*\circ\io_{\V*}(b)=q_*\circ\ph_*(b)$ pour tout $b\in H(X,\V,R)$. 

Pour $a\in H(X,R)$, nous avons $a=\io_{\V*}\circ Sd_{\V*}(a)$, d'o\`u
$$
q_*(a)=q_*\circ\io_{\V*}\circ Sd_{\V*}(a)=q_*\circ\ph_*\circ Sd_{\V*}(a).
$$

Puisque $q_*$ est un isomorphisme, cela entra\^\i ne que $a=\ph_*\circ Sd_{\V*}(a)$. \end{proof}

Si $\ph$ est un endomorphisme d'un espace vectoriel $E$, nous notons $\tr\ph$, quand elle est d\'efinie, la trace au sens de Leray de cet endomorphisme. Toutes les propri\'et\'es de la trace et du nombre de Lefschetz au sens de Leray dont nous avons besoin se trouvent dans \cite{GD}.

Si $\U$ est un recouvrement ouvert d'un espace $X$ et $A$ un sous-ensemble de $X$, la notation $\St(A,\U)$ d\'esigne la r\'eunion de tous les \'el\'ements de $\U$ qui rencontrent $A$. Nous notons $\St(\U)$ le recouvrement form\'e des ensembles $\St(U,\U)$ avec $U\in\U$. Si $Y$ est un sous-ensemble de $X$, la notation $\U|Y$ d\'esigne le recouvrement de $Y$ form\'e des ensembles $U\cap Y$, $U\in\U$.

Nous utiliserons le fait suivant dans la d\'emonstration du lemme 5.

\begin{lem} Soit $F$ un ferm\'e d'un espace m\'etrisable $X$, et soit $\U=\{U_\al\,|\,\al\in A\}$ un recouvrement localement fini de $F$ par des ouverts de $F$. Il existe une famille localement finie $\V=\{V_\al\,|\,\al\in A\}$ d'ouverts de $X$ telle que $V_\al\cap F=U_\al$ pour tout $\al\in A$ et que les nerfs de $\U$ et $\V$ soient naturellement isomorphes. \end{lem}

\begin{proof} Nous pouvons supposer $U_\al\ne\vi$ pour tout $\al\in A$, ce qui nous permet d'identifier $A$ \`a l'ensemble des sommets du nerf $N$ de $\U$. Puisque $\U$ est localement fini et $F$ m\'etrisable, il existe une fonction continue $\psi:F\to|N|$ telle que $U_\al=\psi^{-1}(\St(\al,N))$ pour tout $\al\in A$. Puisque $X$ est m\'etrisable, il existe un voisinage ouvert $W$ de $F$ dans $X$ et une fonction continue $\bar\psi:W\to|N|$ prolongeant $\psi$ ( voir \cite{Hu}, th\'eor\`eme 10.4, page 105). Pour $\al\in A$, soit $W_\al=\bar\psi^{-1}(\St(\al,N))$; c'est un ouvert de $X$ tel que $W_\al\cap F=U_\al$. La famille $\{W_\al\,|\,\al\in A\}$ est localement finie dans $W$ (\cite{Hu}, lemme 10.2 page 102) et les nerfs des familles $\U$ et $\W$ sont naturellement isomorphes, \'etant isomorphes au nerf de la famille $\{\St(\al,N)\,|\,\al\in A\}$. Soit $V$ un voisinage ouvert de $F$ dans $X$ tel que $\ov V\su W$. Alors la famille $\V=\{V_\al\,|\,\al\in A\}$, o\`u $V_\al=W_\al\cap V$ pour tout $\al\in A$ est localement finie dans $X$ et v\'erifie $V_\al\cap F=U_\al$ pour tout $\al$. Puisque les nerfs des familles $\U$ et $\W$ sont naturellement isomorphes et que $U_\al\su V_\al\su W_\al$ pour tout $\al$, les nerfs des familles $\U$ et $\V$ sont naturellement isomorphes. \end{proof}

\section{Une propri\'et\'e des r\'eunions finies de r\'etractes absolus de voisinage alg\'ebriques}
Le r\'esultat que nous d\'emontrerons das cette section entra\^\i ne que, si l'espace $C$ du th\'eor\`eme est m\'etrisable, il est possible d'\og approximer\fg la fonction compacte $f$ par des morphismes de cha\^\i nes dont les images sont de type fini. C'est la clef de la d\'emonstration dans le cas m\'etrisable.

Pour tout complexe simplicial $N$, nous notons $C(N,R$ le complexe des cha\^\i nes orient\'ees de $N$ \`a coefficients $R$. Le module $C_q(N,R)$ des $q$-cha\^\i nes de ce complexe est libre, et a une base en correspondance biunivoque avec les $q$-simplexes de $N$, le g\'en\'erateur correspondant au $q$-simplexe $s$ \'etant un g\'en\'erateur de $C_q(s,R)$. Nous identifions $C_q(N,R)$ au $R$-module libre engendr\'e par les $q$-simplexes de $N$ en fixant une fois pour toutes un g\'en\'erateur de chaque $C_p(s,R)$, que nous notons encore $s$. Si $s$ est un $0$-simplexe, le g\'en\'erateur fix\'e de $C_0(s,R)$ est celui d'augmentation $1$.

Les notions suivantes ont \'et\'e introduites dans \cite{Ca2}.

\begin{defi} Soient $N$ un complexe simplicial, $X$ un espace topologique et $\U$ un recouvrement ouvert de $X$. Une r\'ealisation alg\'ebrique partielle de $N$ dans $X$ relativement \`a $\U$ est la donn\'ee d'un sous-complexe $M$ de $N$ contenant tous les sommets de $N$ et d'un morphisme de cha\^\i nes $\mu:C(M,R)\to S(X,R)$ tel que, pour tout simplexe $s$ de $N$, il existe un \'el\'ement de $\U$ contenant $||\mu(t)||$ pour toute face $t$ de $s$ appartenant \`a $M$. Si $M=N$, la r\'ealisation alg\'ebrique est dite compl\`ete. \end{defi}

\begin{defi} Un espace m\'etrisable $X$ est appel\'e un $R$-r\'etracte absolu de voisinage alg\'ebrique si, pour tout recouvrement ouvert $\U$ de $X$, il existe un recouvrement ouvert $\V$ de $X$ qui est plus fin que $\U$ et tel que, pour tout complexe simplicial $N$, toute r\'ealisation alg\'ebrique partielle de $N$ dans $X$ relativement \`a $\V$ se prolonge en une r\'ealisation compl\`ete de $N$ relativement \`a $\U$. \end{defi}

Dans le lemme suivant, l'anneau $R$ est suppos\'e noeth\'erien.

\begin{lem} Soient $X_1,\dots,X_n$ des ferm\'es d'un espace m\'etrisable $X$ tels que $X=X_1\cup\dots\cup X_n$ et que $X_u$ soit un $R$-r\'etracte absolu de voisinage alg\'ebrique pour tout sous-ensemble non vide $u$ de $\{1,\dots,n\}$ tel que $X_u=\bigcap_{i\in u}X_i\ne\vi$. Alors, pour tout recouvrement ouvert $\U$ de $X$, il existe un recouvrement ouvert $\V$ de $X$ et un morphisme de cha\^\i nes $\ph:S(X,\V,R)\to S(X,R)$ v\'erifiant

$(i)$ pour tout simplexe singulier $\si$ de $S(X,\V,R)$, il existe $U\in\U$ contenant $||\si||$ et $||\ph(\tau)||$ pour toute face $\tau$ de $\si$,

$(ii)$ pour tout compact $D$ de $X$, il existe un voisinage $E$ de $D$ dans $X$ tel que $\ph(S(E,R)\cap S(X,\V,R))$ soit de type fini. \end{lem}

\begin{proof} Pour tout sous-ensemble $u$ de $\{1,\dots,n\}$, nous notons $m(u)$ son cardinal. Soit $k$ le plus grand entier tel qu'il existe un sous-ensemble $u$ de $\{1,\dots,n\}$ avec $m(y)=k$ et $X_u\ne\vi$. Pour $1\le j\le k$, soit $X(j)=\bigcup_{m(u)\ge j}X_u$; c'est un sous-ensemble ferm\'e de $X$.

Soit $\U$ un recouvrement ouvert de $X$. Soit $\U^0$ un recouvrement ouvert de $X$ tel que $\St(\U^0)$ soit plus fin que $\U$. Pour $1\le j\le k$, nous construirons un recouvrement (relativement) ouvert $\U^j$ de $X(j)$ et, pour tout sous-ensemble $u$ de $\{1,\dots,n\}$ tel que $X_u\ne\vi$, nous construirons des recouvrements (relativement) ouverts $\U_u$ et $\W_u$ de $X_u$ de fa\c con que les conditions suivantes soient v\'erifi\'ees.

\medskip \noindent
(1) Si $m(u)=j$, alors $\U_u$ est plus fin que $\U^{j-1}|X_u$ et, pour tout complexe simplicial $N$, toute r\'ealisation alg\'ebrique partielle de $N$ dans $X_u$ relativement \`a $\U_u$ se prolonge en une r\'ealisation alg\'ebrique compl\`ete relativement \`a $\U^{j-1}|X_u$.

\medskip \noindent
(2) $\St(\W_u)$ est plus fin que $\U_u$.

\medskip \noindent
(3) $\U^j$ est plus fin que $\U^{j-1}|X(j)$ pour tout $0<j\le k$.

\medskip \noindent
(4) $\U^j|X_u$ est plus fin que $\W_u$ pour tout $u\su\{1,\dots,n\}$ tel que $m(u)=j$.

\medskip
Si $\U^{j-1}$ est construit, le fait que $X_u$ est un $R$-r\'etracte absolu de voisinage alg\'ebrique nous permet de trouver $\U_u$ v\'erifiant (1), et l'existence de $\W_u$ r\'esulte de la paracompacit\'e de $X_u$. Une fois les $\W_u$ construits pour tout $u$ tel que $m(u)=j$, tout point $x$ de $X(j)$ a un voisinage $O(x)$ tel que $O(x)\cap X_u$ soit contenu dans un \'el\'ement de $\W_u$ pour tout $u$ tel que $m(u)=j$ et $x\in X_u$ et que $O(x)\cap X_u=\vi$ si $x\notin X_u$. Prenons alors pour $\U^j$ un recouvrement ouvert de $X(j)$ plus fin que $\U^{j-1}|X(j)$ et que $\{O(x)\,|\,x\in X(j)\}$.

Nous construirons maintenant, par r\'ecurrence descendante, des ensembles $A_k\su\dots\su A_1$ et, pour $k\ge j\ge1$, un recouvrement ouvert localement fini $\V_j=\{V_{\al,j}\,|\,\al\in A_j\}$ de $X_j$ qui est plus fin que $U^j$. Nous construirons aussi, pour tout sous-ensemble non vide $u$ de $\{1,\dots,n\}$ tel que $X_u\ne\vi$, un sous-ensemble $A'_u$ de $A_{m(u)}$ de fa\c con que $A_j$ soit la r\'eunion disjointe des $A'_u$ avec $m(u)\ge j$ et que la condition suivante soit v\'erifi\'ee.

\medskip \noindent
(5) Si $i\notin u$, alors $V_{\al,j}\cap X_i=\vi$ pour tout $\al\in A'_u$ et tout $j\le m(u)$.

\medskip
$X(k)$ est la r\'eunion disjointe des $X_u\ne\vi$ tels que $m(u)=k$. Pour chaque tel $u$, soit $\V_u=\{V_{\al,k}\,|\,\al\in A'_u\}$ un recouvrement ouvert localement fini de $X_u$ plus fin que $\U^k$. Soit $A_k$ la r\'eunion disjointe des $A'_u$ avec $m(u)=k$, et soit $\V_k=\bigcup_{m(u)=k}\V_u$.

Soit $1\le j<k$, et supposons $A_{j+1}$ et $\V_{j+1}$ construits. Le lemme 4 nous permet de trouver une famille $\wh\V_{j+1}=\{V_{\al,j}\,|\,\al\in A_{j+1}\}$ d'ouverts de $X(j)$, localement finie dans $X(j)$,  telle que $V_{\al,j}\cap X(j+1)=V_{\al,j+1}$ pour tout $\al\in A_{j+1}$, et dont le nerf est naturellement isomorphe au nerf de $\V_{j+1}$. Puisque $\V_{j+1}$ est plus fin que $\U^{j+1}$, qui est plus fin que $\U^j$, nous pouvons, quitte \`a diminuer les $V_{\al,j}$, supposer que $\wh\V_{j+1}$ est plus fine que $\U^j$. Si $i\notin u$ et si $m(u)\ge j+1$, alors $V_{\al,j+1}$ est disjoint du ferm\'e $X_i$ pour tout $\al\in A'_u$; quitte \ a diminuer $V_{\al,j}$, nous pouvons aussi supposer qu'il est disjoint de $X_i$. Pour tout $v\su\{1,\dots,n\}$ tel que $m(v)=j$, recouvrons le ferm\'e $X_v\sm \bigcup\{V_{\al,j}\,|\,\al\in A_{j+1}\}\su X_v\sm X(j+1)$ par une famille localement finie $\V_v=\{V_{\al,j}\,|\,\al\in A'_v\}$ d'ouverts de $X_v$ contenus dans $X_v\sm X(j+1)$ et dans des \'el\'ements de $\U^j$. Comme $X_v\sm X(j+1)$ est ouvert dans $X(j)$, les $V_{\al,j}$ avec $\al\in A'_v$ sont ouverts dans $X(j)$; en outre, ils v\'erifient (5). Soit $A_j$ la r\'eunion disjointe de $A_{j+1}$ et des $A'_v$ avec $m(v)=j$, et soit $\V_j=\wh\V_{j+1}\cup\bigcup_{m(v)=j}\V_v$.

Posons $A=A_1$, $\V=\V_1$ et $V_\al=V_{\al,1}$ pour $\al\in A$. Nous supposons que les $V_\al$ sont non vides, ce qui nous permet d'identifier $A$ \ a l'ensemble des sommets du nerf $K$ de $\V$. Pour $u\su\{1,\dots,n\}$ tel que $X_u\ne\vi$, soit $A_u=\bigcup_{u\su v}A'_v$. Notons que

\medskip \noindent
(6) pour tout simplexe $s=\ag\al_0,\dots,\al_q\ad$ de $K$, il existe un unique sous-ensemble $u_s$ de $\{1,\dots,n\}$ tel que $s\cap A'_{u_s}\ne\vi$ et $s\su A_{u_s}$; en outre, si $j_s=m(u_s)$, alors $V_{\al_0}\cap\dots\cap V_{\al_q}\cap X_{u_s}\ne\vi$.

\medskip
En effet, soit $j_s$ le plus grand entier tel que $\{\al_0,\dots,\al_q\}\su A_{j_s}$. Il r\'esule du choix des familles $\wh\V_{j+1}$ que $V_{\al_0,j_s}\cap\dots\cap V_{\al_q,j_s}\ne\vi$. Il existe $p$ tel que $\al_p\in A_{j_s}\sm A_{j_s+1}$ ($A_{k+1}=\vi$). Si $u_s$ est le sous-ensemble de $\{1,\dots,n\}$ tel que $\al_p\in A'_{u_s}$, il r\'esulte de (5) que $u_s$ ne d\'epend pas du choix de l'\'el\'ement $\al_p$ tel que $\al_p\in A_{j_s}\sm A_{j_s+1}$ et que $\al_a\in A_{u_s}$ pour $1\le a\le q$. La deuxi\`eme affirmation de (6) r\'esulte du fait que, $V_{\al_p,j_s}\su X_{u_s}\sm X(j_s+1)$.

Puisque $\V$ est localement fini, si $\si$ est un simplexe singulier de $S(X,\V,R)$, l'ensemble des $\al\in A$ tels que $||\si||\su V_\al$ est un simplexe $s_\si$ de $K$. Si $\tau$ est une face de $\si$, alors $s_\si\su s_\tau$. Soit $\fT_\si$ l'ensemble des suites finie $\ft=(\si_0,\dots,\si_q)$ de faces de $\si$ v\'erifiant $\si_0<\dots<\si_q$; pour une telle $\ft$, $\ag[s_{\si_q}],\dots,[s_{\si_0}]\ad$ est un simplexe $s'_\ft$ de $K'$. Soit $K'_\si$ le sous-complexe de $K'$ form\'e des simplexes $s'_\ft$ o\`u $\ft$ parcourt $\fT_\si$.

Toute suite $\ft\in\fT_\si$ est une sous-suite d'un \'el\'ement $\ft'=(\si_0,\dots,\si_p)$ de $\fT_\si$ tel que $\si_p=\si$, donc $s'_\ft$ est face d'un simplexe $s'_{\ft'}$ de $K'_\si$ dont $[s_\si]$ est un sommet. Il en r\'esulte que $K'_\si$ est un c\^one de sommet $[s_\si]$ ou est r\'eduit \`a $[s_\si]$, donc $C(K'_\si,R)$ est acyclique. Si $\si<\tau$ sont deux simplexes singuliers de $S(X,\V,R)$, alors $\fT_\si\su\fT_\tau$, donc $K'_\si\su K'_\tau$. Proc\'edant par r\'ecurrence sur la dimension de $\si$, l'acyclicit\'e des complexes $C(K'_\si,R)$ nous permet de trouver un morphisme de cha\^\i nes $\zeta:S(X,\V,R)\to C(K',R)$ tel que $\zeta(\si)\in C(K'_\si,R)$ pour tout simplexe singulier $\si\in S(X,\V,R)$.

Pour $1\le j\le k$, soit $K_j$ le sous-complexe plein de $K$ engendr\'e par les sommets appartenant \`a $A_j$; posons $K_{k+1}=\vi$. Pour $1\le j\le k+1$, soit $L_j={K'}^{(0)}\cup K'_j$, o\`u ${K'}^{(0)}$ est le $0$-squelette de $K'$. Pour $u\su\{1,\dots,n\}$, soit $K_u$ le sous-complexe plein de $K$ engendr\'e par les sommets appartenant \`a $A_u$.

Nous construirons maintenant un morphisme de cha\^\i nes $\xi:C(K',R)\to S(X,R)$ v\'erifiant

\medskip \noindent (7) Pour tout simplexe $s'=\ag[s_0],\dots,[s_p]\ad$ de $K'$, il existe $U\in \U^{j(s_0)-1}$ tel que $\xi(C(s',R))\su S(X_{u(s_0)}\cap U,R)$ ($u(s_0)$ et $j(s_0)$ comme dans (6)).

\medskip
Pour construire $\xi$, fixons, pour tout simplexe $s=\ag\al_0,\dots,\al_q\ad$ de $K$, un point $x_s\in V_{\al_0}\cap\dots\cap V_{\al_q}\cap X_{u_s}$. Nous d\'efinissons la restriction $\xi_{k+1}$ de $\xi$ \`a $C(L_{k+1},R)=C({K'}^{(0)},R)$ en envoyant le g\'en\'erateur de $C([s],R)$ sur le $0$-simplexe singulier $x_s$ pour tout simplexe $s$ de $K$. Nous construirons ensuite les restrictions $\xi_j$ de $\xi$ \`a $C(L_j,R)$ par r\'ecurrence descendante. Soit $1\le j\le k$, et supposons $\xi_{j+1}$ construite. Pour tout $u\su\{1,\dots,n\}$ tel que $m(u)=j$, soit $M_u=K'_u\cap L_{j+1}$; c'est un sous-complexe de $K'_u$ contenant tous les sommets de $K'_u$. Nous allons v\'erifier que la restriction $\bar\xi_u$ de $\xi_{j+1}$ \`a $M_u$ est une r\'ealisation alg\'ebrique partielle de $K'_u$ dans $X_u$ relativement \`a $\U_u$. D'apr\`es (1), $\bar\xi_u$ se prolongera en une r\'ealisation alg\'ebrique compl\`ete $\xi_u$ relativement \`a $\U^{j-1}|X_u$. Comme $L_j=L_{j+1}\cup\bigcup_{m(u)=j}K'_u$ et $K'_{u_1}\cap K'_{u_2}\su L_{j+1}$ si $m(u_1)=m(u_2)=j$ et $u_1\ne u_2$, nous pourrons alors d\'efinir $\xi_j$ par $\xi_j|C(L_{j+1},R)=\xi_{j+1}$ et $\xi_j|C(K'_u,R)=\xi_u$ pour tout $u$ tel que $m(u)=j$. La condition (7) pour les simplexes de $K'_u$ r\'esulte de la d\'efinition d'une r\'ealisation alg\'ebrique compl\`ete relativement \`a $\U^{j-1}|X_u$.

Fixons $u$ tel que $m(u)=j$. Soit $s'=\ag[s_0],\dots,[s_p]\ad$ un simplexe de $K'_u$, et soit $\al_0$ un sommet de $s_0$. Pour $1\le a\le p$, nous avons $u\su u_{s_a}$, d'o\`u $X_{u_{s_a}}\su X_u$, et le point $x_{s_a}$ appartient \`a $V_{\al_0}\cap X_u=V_{\al_0,j}\cap X_u$. Puisque $\V_j$ est plus fin que $\U^j$, il existe $U_0\in\U^j$ tel que $x_{s_a}\in U_0\cap X_u$ pour tout $a\le p$. Soit $\bar s=\ag[s_{a_1}],\dots,[s_{a_r}]\ad$ une face de $s'$ contenue dans $L_{j+1}$. La condition (7) \'etant v\'erifi\'ee au rang $j+1$, il existe $U_{\bar s}\in\U^{j(s_{a_1})-1}$ tel que $||\xi(\hat s)||\su U_{\bar s}\cap X_{u(s_{a_1})}\su U_{\bar s}\cap X_u$ pour toute face $\hat s$ de $\bar s$; en particulier, $U_{\bar s}$ contient $x_{s_{a_1}}$. Puisque $\bar s$ est contenu dans $L_{j+1}$, nous avons $j(s_{a_1})>j$, et (3) garantit l'existence d'un \'el\'ement $U_{\bar s}^j$ de $\U^j$ contenant $U_{\bar s}$. Puisque $x_{s_{a_1}}\in U_{\bar s}\cap U_0\cap X_u$, nous avons $U_{\bar s}^j\cap X_u\su\St(U_0\cap X_u,\U^j|X_u)$, donc $\St(U_0\cap X_u,\U^j|X_u)$ contient $||\xi(\bar s)||$ pour toute face $\bar s$ de $s'$ appartenant \`a $L_{j+1}$. Les conditions (4) et (2) garantissent que $\St(U_0\cap X_u,\U^j|X_u)$ est contenu dans un \'el\'ement de $\U_u$, donc $\bar\xi_u$ est bien une r\'ealisation alg\'ebrique partielle dans $X_u$ relativement \ a $\U_u$. La construction de $\xi$ est ainsi achev\'ee.

Soit $\ph=\xi\circ\zeta:S(X,\V,R)\to S(X,R)$.

Soit $\si$ un simplexe singulier de $S(X,\V,R)$. Il existe $\al_0\in A$ tel que $||\si||\su V_{\al_0}$. Pour toute face $\tau$ de $\si$, $\al_0$ est un sommet de $s_\tau$, ce qui entra\^\i ne que, pour tout $\ft\in\fT_\si$, le point $x_{s_\ft}$ appartient \`a $V_{\al_0}$. Il r\'esulte de (7) et (3) qu'il existe, pour tout simplexe $s'_\ft$ de $K'_\si$, un \'el\'ement $U_\ft$ de $\U^0$ tel que $\xi(C(s'_\ft,R))\su S(U_\ft,R)$. En particulier, si $[s_\tau]$ est un sommet de $s_\ft$ ($\tau$ est une face de $\si$), $U_\ft$ contient le point $x_{s_\tau}$ de $V_{\al_0}$, donc est contenu dans $\St(V_{\al_0},\U^0)$. Nous avons donc $\xi(C(K'_\si,R))\su S(\St(V_{\al_0},\U^0),R)$. Puisque $\zeta(\tau )$ appartient \`a $C(K'_\si,R)$ pour toute face $\tau$ de $\si$, $\St(V_{\al_0},\U^0)$ contient $||\si||\cup||\ph(\tau)||$ pour toute face $\tau$ de $\si$. Comme $\V$ est plus fin que $\U^0$ et $\St(\U^0)$ plus fin que $\U$, la condition (i) est v\'erifi\'ee.

Soit $D$ un compact de $X$. Puisque $\V$ est localement fini, il existe un voisinage $E$ de $D$ dans $X$ tel que l'ensemble $A_E$ des $\al\in A$ pour lesquels $V_\al\cap E\ne\vi$ est fini. Soit $K_E$ le sous-complexe plein de $K$ engendr\'e par les sommets appartenant \`a $A_E$. Pour tout simplexe singulier $\si$ appartenant \`a $S(E,R)\cap S(X,\V,R)$, $s_\tau$ appartient \`a $K_E$ pour toute face $\tau$ de $\si$, donc $K'_\si$ est contenu dans $K'_E$. Il en r\'esulte que $\ph(S(E,R)\cap S(X,\V,R))$ est contenu dans $\xi(C(K'_E,R))$. Comme $K'_E$ est un complexe fini, $C(K'_E,R)$ est de type fini, et il en est de m\^eme de tout sous-complexe de $\xi(C(K'_E,R))$ puisque $R$ est noeth\'erien.
\end{proof}

\section{D\'emonstration du th\'eor\`eme}
Dans toute cette section, nous utiliserons exclusivement des coefficients rationnels pour l'homologie et la cohomologie.

Les r\'esultats des sections 2 et 3 nous permettent de traiter le cas o\`u l'espace $C$ du th\'eor\`eme est m\'etrisable, cas qui est la partie la plus d\'elicate de la d\'emonstration.

\begin{lem} Le th\'eor\`eme est vrai quand $C$ est m\'etrisable. \end{lem}

\begin{proof} D'apr\`es le lemme 1, $C$ a le type d'homotopie d'un complexe fini, donc $\La(f)$ est d\'efini et nous pouvons utiliser l'homologie singuli\`ere pour le calculer. Soit $\U_0$ un recouvrement ouvert de $C$ v\'erifiant la conclusion du lemme 3. 

Supposons que $f$ n'ait pas de point fixe. Nous pouvons alors trouver un recouvrement ouvert $\U_1$ de $C$ tel que $U\cap \St(f(U),\U_1)=\vi$ pour tout $U\in\U_1$.

Soit $\U$ le recouvrement ouvert de $C$ form\'e des ensembles $U_0\cap U_1$ avec $U_0\in\U_0$ et $U_1\in\U_1$. Soit $K$ un compact de $C$ tel que $f(C)\su K$. Puisque tout convexe m\'etrisable est un r\'etracte absolu de voisinage alg\'ebrique (\cite{Ca2} th\'eor\`eme 8), le lemme 5 s'applique et nous fournit un recouvrement ouvert $\V$ et un morphisme de cha\^\i nes $\ph:S(C,\V,\Bbb Q)\to S(C,\Bbb Q)$ v\'erifiant les conditions (i) et (ii) de ce lemme relativement \`a $\U$.

Soit $\psi=\ph\circ Sd_\V\circ f_\#:S(C,\Bbb Q)\to S(C,\Bbb Q)$. Puisque $\U$ est plus fin que $\U_0$, la condition (i) du lemme 5 et le lemme 3 entra\^\i nent que $\psi_*=\ph_*\circ Sd_{\V*}\circ f_*=f_*$, donc $\La(\psi_*)=\La(f)$. Nous avons $f_\#(S(C,\Bbb Q))\su S(K,\Bbb Q)$, donc $Sd_\V\circ f_\#(S(C,\Bbb Q))\su S(K,\Bbb Q)\cap S(C,\V,\Bbb Q)$, et la condition (ii) du lemme 5 garantit que l'image de $\psi$ est de type fini, donc $\La(\psi)$ est d\'efini. Alors $\La(\psi_*)=\La(\psi)$ et, pour achever la d\'emonstration du lemme, il suffit de montrer que $\La(\psi)=0$.

Les simplexes singuliers de $S(C,\Bbb Q)$ forment une base de ce complexe, et les coordonn\'ees que nous utilisons sont relatives \`a cette base. Soit $\fS$ l'ensemble des simplexes singuliers $\si$ pour lesquels il existe $c\in S(C,\Bbb Q)$ tel que $\psi(c)$ ait une coordonn\'ee non nulle sur $\si$, et soit $E$ le sous-complexe de $S(C,\Bbb Q)$ engendr\'e par $\fS$. Comme l'image de $\psi$ est de type fini, $\fS$ est fini, donc $E$ est de type fini. Comme $E$ contient l'image de $\psi$, nous avons $\psi(E)\su E$, et si $\psi':E\to E$ est le morphisme induit par $\psi$, alors $\La(\psi)=\La(\psi')$. Pour prouver que $\La(\psi)=0$, il suffit de montrer que si $\psi'_q$ est la composante de degr\'e $q$ de $\psi'$, alors $\tr\psi'_q=0$ pour tout $q\ge0$. Pour cela, il suffit de v\'erifier que, pour tout simplexe $\si\in \fS$, la coordonn\'ee de $\psi(\si)$ sur $\si$ est nulle, ce que nous ferons en montrant que $||\psi(\si)||\cap||\si||=\vi$ pour tout $\si\in\fS$.

Pour tout simplexe singulier $\tau$ de $S(C,\Bbb Q)$, $Sd_\V\circ f_\#(\tau)$ est une combinaison lin\'eaire $\sum\la_i\tau_i$, o\`u $||\tau_i||$ est contenu dans un \'el\'ement de $\V$. Alors $\psi(\tau)=\sum\la_i\ph(\tau_i)$, et la condition (i) du lemme 5 entra\^\i ne que $||\ph(\tau_i)||$ est contenu dans un \'el\'ement de $\U$. Par suite, pour tout simplexe singulier $\si\in\fS$, il existe un \'el\'ement de $\U$ contenant $||\si||$. Par d\'efinition de $\U$, il existe $U_1\in\U_1$ tel que $||\si|| \su U\su U_1$. Soit $Sd_\V\circ f_\#(\si)=\sum_{a=1}^m\mu_a\si_a$, o\`u $||\si_a||$ est contenu dans un \'el\'ement de $\V$ et dans $||f_\#(\si)||\su f(U_1)$. Pour tout $a\le m$, il existe, d'apr\`es (i) du lemme 5, un \'el\'ement $U'_a$ de $\U_1$ contenant $||\si_a||\cup||\ph(\si_a)||$. Puisque $||\si_a||\su f(U_1)\cap U'_a$, nous avons $U'_a\su\St(f(U_1),\U_1)$, et il en r\'esulte que $||\psi(\si)||\su\bigcup_{a=1}^m||\ph(\si_a)||\su\St(f(U_1),\U_1)$, donc $||\psi(\si)||$ est disjoint de $U_1\supset||\si||$ d'apr\`es le choix de $\U_1$.
\end{proof}

Le passage du cas m\'etrisable au cas g\'en\'eral du th\'eor\`eme se fait par un m\'ecanisme syst\'ematique d\'ecrit dans \cite{Ca2}, qui utilise deux ingr\'edients: le convexe libre associ\'e \`a un espace compact, et la repr\'esentation des compacts non m\'etrisables comme limites de syst\`emes projectifs particuliers de compacts m\'etrisables.

A tout espace compact $X$ est associ\'e un espace convexe libre contenant $X$ et caract\'eris\'e par la propri\'et\'e universelle suivante: pour toute fonction continue $f$ de $X$ dans un sous-ensemble convexe $C$ d'un e.v.t., il existe une unique application affine continue $F:C(X)\to C(X)$ telle que $F|X=f$. Ce convexe libre peut se construire comme suit. Pour $m\ge0$, soit $C_m(X)$ l'ensemble des combinaisons convexes formelles $\sum_{i=0}^m\la_ix_i$ de $m+1$ points de $X$ (pas n\'ec\'essairement distincts; $\la_i\ge0$ et $\sum_{i=0}^m\la_i=1$). $C_m(X)$ est muni de la topologie quotient du produit $X^{m+1}\ti\Dt_m$ par l'application $\mu_m$ d\'efinie, pour $x_0,\dots,x_m$ dans $X$ et $(\la_0,\dots,\la_m)\in\Dt_m$, par
$$
\mu_m(x_0,\dots,x_m,(\la_0,\dots,\la_m))=\sum_{i=0}^m\la_ix_i.
$$

$C_m(X)$ est un compact, s'identifie naturellement \`a un sous-espace de $C_{m+1}(X)$, et $C(X)=\bigcup_{m=0}^\ii C_m(X)$ est la limite inductive de la suite croissante de compacts $\{C_m(X)\}$. Le compact $X$ est identifi\'e au sous-espace $C_0(X)$ de $C(X)$. Notons que $C(X)$ est un sous-ensemble convexe ferm\'e de l'espace vectoriel topologique libre $E(X)$ engendr\'e par $X$, que nous avons utilis\'e dans \cite{Ca1}.

Si $Y$ est un sous-ensemble ferm\'e de $X$, alors $C(Y)$ s'identifie naturellement \ a un sous-espace ferm\'e de $C(X)$. Le lemme suivant est un cas particulier du th\'eor\`eme.

\begin{lem} Soient $X_1,\dots,X_n$ des ferm\'es d'un compact $X$ tels que $X=X_1\cup\dots\cup X_n$, et soit $K=C(X_1)\cup\dots\cup C(X_n)\su C(X)$. Si $f:K\to K$ est une fonction continue telle que $f(K)\su X$, alors $\La(f)$ est d\'efini. Si $\La(f)\ne0$, alors $f$ a un point fixe.  \end{lem}

Pour tout compact $X$, $C(X)$ est r\'egulier et $\si$-compact, donc paracompact. Si $K$ est comme dans le lemme 7, il est paracompact, donc, comme nous l'avons remarqu\'e apr\`es le lemme 1, $\cH(K,\Bbb Q)$ est de type fini et $\La(f)$ est d\'efini.

L'utilisation de $C(X)$ est justifi\'ee par la remarque suivante.

\begin{lem} Si le lemme 7 est vrai, il en est de m\^eme du th\'eor\`eme. \end{lem}

\begin{proof} Soient $C=C_1\cup\dots\cup C_n$ comme dans l'\'enonc\'e du th\'eor\`eme. Soit $X$ un sous-ensemble compact de $C$ tel que $f(C)\su X$. Pour $1\le i\le n$, soit $X_i=C_i\cap X$. Les $X_i$ sont des ferm\'es de $X$ tels que $X=X_1\cup\dots\cup X_n$. Soit $K=C(X_1)\cup\dots\cup C(X_n)\su C(X)$. L'inclusion $i:X\to C$ se prolonge en une application continue affine $r$ de $C(X)$ dans l'enveloppe convexe de $C$, qui envoie $C(X_i)$ dans $C_i$ pour tout $i$, donc induit une fonction continue $h:K\to C$.

Soit $\bar f:C\to X$ la fonction induite par $f$, et soit $j$ l'inclusion de $X$ dans $K$. Puisque $r|X=j$, nous avons $h\circ j\circ\bar f=f$. Soit $g=j\circ\bar f\circ h:K\to K$. L'image de $g$ est contenue dans $X$ et $g(x)=f(x)$ pour tout $x\in X$. Alors $\La(g)$ est d\'efini et, si le lemme 7 est vrai, $g$ a un point fixe si $\La(g)\ne0$. Mais si $\La(g)=\La(j\circ\bar f\circ h)$ est d\'efini, il en est de m\^eme de $\La(h\circ j\circ\bar f)=\La(f)$, et $\La(g)=\La(f)$ (voir \cite{GD}, p. 417). Puisque tout point fixe de $g$ appartient \ a $X$, c'est aussi un point fixe de $f$, d'o\`u le lemme. \end{proof}

Lorsque le compact $X$ est m\'etrisable, le lemme 7 r\'esulte du lemme 6.

\begin{lem} Le lemme 7 est vrai quand $X$ est m\'etrisable. \end{lem}

\begin{proof} Si $X$ est m\'etrisable, les lemmes 2.1 et 2.2 de \cite{Ca1} nous permettent de touver une topologie vectorielle m\'etrisable $\tau$ sur $E(X)$, moins fine que la topologie libre et telle que $C(X_i)$ soit $\tau$-ferm\'e dans $C(X)$ pour $i=1,\dots,n$ et que $f:(K,\tau)\to X$ soit continue. Consid\'erons le diagramme commutatif

$$\begin{CD}
\cH((K,\tau),\Bbb Q)@>\bar{f}^*>>\cH((K,\tau),\Bbb Q)\\
@V{j^*}VV @VV{j^*}V\\
\cH(K,\Bbb Q)@>f^*>>\cH(K,\Bbb Q)
\end{CD}$$
o\`u $\bar f=f$ et $j:K\to(K,\tau)$ est l'identit\'e. D'apr\`es le lemme 2, $j^*$ est un isomorphisme, donc $\La(f)=\La(\bar f)$. Si $\La(f)\ne0$, alors $\bar f=f$ a un point fixe d'apr\`es le lemme 6. 
\end{proof}

Le cas g\'en\'eral du lemme 8 se d\'eduit du cas m\'etrisable en repr\'esentant $X$ comme limite d'un syst\`eme projectif particulier de compacts m\'etrisables.

Un syst\`eme projectif d'espaces topologiques sur un ensemble ordonn\'e filtrant $A$ sera not\'e $\bS=(X_\al,p_\al^\bt,A)$ (les $X_\al$ sont des espaces topologiques et, pour $\al\le\bt$, $p_\al^\bt:X_\bt\to X_\al$ est continue). Nous notons $\lim \bS$ la limite projective de ce syst\`eme et $p_\al$ la projection de $\lim \bS$ dans $X_\al$. Si $B$ est un sous-ensemble filtrant de $A$, nous notons $\bS|B=(X_\al,p_\al^\bt,B)$ le syst\`eme obtenu en restreignant l'ensemble des indices \`a $B$. Si $B$ est cofinal dans $A$, nous identifions naturellement $\lim(\bS|B)$ \`a $\lim \bS$. Si $\bS_1=(X_\al,p_\al^\bt,A)$ et $\bS_2=(Y_\al,q_\al^\bt,A)$ sont deux syst\`emes projectifs ayant le m\^eme ensemble d'indices, un morphisme de $\bS_1$ dans $\bS_2$ est une famille de fonctions continues $f_\al:X_\al\to Y_\al$, $\al\in A$, telle que $f_\al p_\al^\bt=q_\al^\bt f_\bt$ pour $\al\le\bt$. Un tel morphisme induit une fonction continue $\lim(f_\al)$ de $\lim \bS_1$ dans $\lim \bS_2$.

Un syst\`eme projectif $\bS=(X_\al,p_\al^\bt,A)$ est appel\'e un $\om$-syst\`eme projectif s'il v\'erifie les conditions suivantes.
\begin{itemize}
\item[($*$)] Tout sous-ensemble d\'enombrable totalement ordonn\'e de $A$ a une borne sup\'erieure.

\item[($**$)] Pour tout sous-ensemble totalement ordonn\'e $B$ de $A$ admettant une borne sup\'erieure $\bt$, la fonction $\lim_{\al\in B}p_\al^\bt:X_\bt\to\lim(\bS|B)$ est un hom\'eomorphisme.

\item[($***$)] Chaque $X_\al$ a un base d\'enombrable.
\end{itemize}

Un sous-ensemble $B$ d'un ensemble filtrant $A$ est dit ferm\'e si, pour tout sous-ensemble totalement ordonn\'e $C\su B$ admettant une borne sup\'erieure dans $A$, cette borne sup\'erieure appartient \`a $B$. Si $\bS=(X_\al,p_\al^\bt,A)$ est un $\om$-syst\`eme projectif et si $B$ est un sous-ensemble cofinal et ferm\'e de $A$, alors $\bS|B$ est aussi un $\om$-syst\`eme projectif.

Le passage du cas m\'etrisable au cas g\'en\'eral du lemme 7 est permis par le fait suivant, qui est un cas particulier du th\'eor\`eme 1.3.4 de \cite{Ch}.

\begin{lem} Soient $\bS_1=(X_\al,p_\al^\bt,A)$ et $\bS_2=(Y_\al,q_\al^\bt,A)$ deux $\om$-syst\`emes de compacts sur un m\^eme ensemble d'indices $A$, et soit $f$ une fonction continue de $\lim \bS_1$ dans $\lim \bS_2$. Si, pour tout $\al\in A$, la projection $p_\al:\lim \bS_1\to X_\al$ est surjective, alors il existe un sous-ensemble cofinal et ferm\'e $B$ de $A$ et un morphisme de $\bS_1|B$ dans $\bS_2|B$ dont la limite est $f$. \end{lem}

\noindent
{\it D\'emonstration du lemme 7 dans le cas non m\'etrisable.} Nous utilisons la m\'ethode expliqu\'ee dans \cite{Ca3} dans le cas des espaces ULC. Nous repr\'esentons $X$ comme limite d'un $\om$-syst\`eme projectif $\bS=(X_\al,p_\al^\bt,A)$ de compacts m\'etrisables tel que les projections $p_\al:X\to X_\al$ soient surjectives.

Pour $m\ge0$, soit $K_m=\bigcup_{i=1}^nC_m(X_i)=K\cap C_m(X)$. Pour $\al\in A$ et $1\le i\le n$, soit $X_{i,\al}=p_\al(X_i)$. Posons $K_\al=\bigcup_{i=1}^n  C(X_{i,\al})\su C(X_\al)$ et, pour $m\ge0$, $K_{\al,m}=K_\al\cap C_m(X)=\bigcup_{i=1}^nC_m(X_{i,\al})$. Les fonctions $p_\al$ induisent des fonctions $\tilde p_\al:K\to K_\al$ et $p_{\al,m}:K_m\to K_{\al,m}$. Pour $\al\le\bt$, les fonctions $p_\al^\bt$ induisent des fonctions $\tilde p_\al^\bt:K_\bt\to K_\al$ et $p_{\al,m}^\bt:K_{\bt,m}\to K_{\al,m}$. Les foncteurs $C_m(\ \cdot\ )$ commutent aux limites projectives; il est facile d'en d\'eduire que $K_m(\bS)=(K_{\al,m},p_{\al,m}^\bt,A)$ est un $\om$-syst\`eme projectif de limite $K_m$. La surjectivit\'e des fonctions $p_\al$ entra\^\i ne que $p_{\al,m}:K_m\to K_{\al,m}$ est surjective.

Nous avons $K_0(\bS)=\bS$; appliquant le lemme 10 \`a la restriction de $f$ \`a $K_0=X$, nous obtenons un sous-ensemble $A_0$, cofinal et ferm\'e dans $A$, et un morphisme $(f_\al^0)_{\al\in A_0}:K_0(\bS)|A_0\to \bS|A_0$ de limite $f|X$. Comme $K_1(\bS)|A_0$ et $\bS|A_0$ sont des $\om$-syst\`emes projectifs, nous pouvons appliquer le lemme 10 \`a la restriction de $f$ \`a $K_1$ pour obtenir un sous-ensemble $A_1$ de $A_0$, cofinal et ferm\'e dans $A_0$, donc aussi dans $A$, et un morphisme $(f_\al^1)_{\al\in A_1}: K_1(\bS)|A_1\to\bS|A_1$ de limite $f|K_1$. Inductivement, nous construisons ainsi une suite d\'ecroissante $\{A_m\}$ de sous-ensembles cofinaux et ferm\'es de $A$ et des morphismes $(f_\al^m)_{\al\in A_m}: K_m(\bS)|A_m\to\bS|A_m$ de limite $f|K_m$. La condition ($*$) garantit que $B=\bigcap_{m=1}^\ii$ est cofinal dans $A$, et la surjectivit\'e des fonctions $p_{\al,m}$ garantit que, pour tout $\al\in B$, $f_\al^m|K_{\al,k}=f_\al^k$ pour $k\le m$, ce qui nous permet de d\'efinir une fonction continue $\bar f_\al:K_\al\to X_\al$ par $\bar f_\al|K_{\al,m}=f_\al^m$ pour tout $m$. Soit $f_\al:K_\al\to K_\al$ la fonction induite par $\bar f_\al$. Ces fonctions v\'erifient $f_\al\tilde p_\al^\bt=\tilde p_\al^\bt f_\bt$ pour $\al\le\bt$ dans $A$.

Le nerf du syst\`eme $\CX=(X_1,\dots,X_n)$ (resp. $\CX_\al=(X_{1,\al},\dots,X_{n,\al})$ est naturellement isomorphe au nerf du syst\`eme $C(\CX)=(C(X_1),\dots,C(X_n))$ (resp. $C(\CX_\al)=(C(X_{1,\al}),\dots,C(X_{n,\al}))$). Puisque le compact $X$ est la limite projective du syst\`eme $\bS|B$, le nerf de $\CX$ est naturellement isomorphe au nerf de $\CX_\al$ si $\al$ est assez grand. Quitte \`a remplacer $B$ par un sous-ensemble cofinal, nous pouvons donc supposer que les nerfs des syst\`emes $\CX$ et $\CX_\al$ sont naturellement isomorphes pour tout $\al\in B$. Les nerfs des syst\`emes $C(\CX)$ et $C(\CX_\al)$ sont alors naturellement isomorphes, et il r\'esulte du lemme 2 que $\tilde p_\al^*:\cH(K_\al,\Bbb Q)\to \cH(K,\Bbb Q)$ est un isomorphisme. Pour $\al\in B$, consid\'erons le diagramme commutatif
$$
\begin{CD}
\cH(K,\Bbb Q)@>f^*>>\cH(K,\Bbb Q)\\
@A{\tilde p_\al^*}AA @AA{\tilde p_\al^*}A\\
\cH(K_\al,\Bbb Q)@>f^*>>\cH(K_\al,\Bbb Q)
\end{CD}
$$

Nous avons remarqu\'e plus haut que $\La(f)$ et $\La(f_\al)$ sont d\'efinis. Puisque $\tilde p_\al^*$ est un isomorphisme, nous avons $\La(f)=\La(f_\al)$. Si $\La(f)\ne0$, alors l'ensemble $F_\al$ des points fixes de $f_\al$ n'est pas vide d'apr\`es le lemme 9. Puisque $f_\al\tilde p_\al^\bt=\tilde p_\al^\bt f_\bt$, nous avons $p_\al^\bt(F_\bt)\su F_\al$ pour $\al\le\bt$ dans $B$, donc $(F_\al,p_\al^\bt|F_\bt,B)$ est un syst\`eme projectif de compacts non vides. La limite projective de ce syst\`eme est un sous-ensemble non vide de $\lim \bS|B=X$, et tout point de ce sous-ensemble est un point fixe de $f$. \hfill $\Box$

\end{document}